\documentclass[12pt]{article}
\usepackage[margin=1in]{geometry} 
\usepackage{amsmath,amsthm,amssymb,amsfonts}

\usepackage{etoolbox}
\usepackage{hyperref}

\makeatletter
\newcounter{author}
\renewcommand*\author[1]{%
	\stepcounter{author}%
	\ifnum\c@author=1
	\gdef\@author{#1}%
	\else
	\xdef\@author{\unexpanded\expandafter{\@author\and#1}}%
	\fi
	\csgdef{author@\the\c@author}{#1}}
\newcommand*\email[1]{%
	\csgdef{email@\the\c@author}{#1}}
\newcommand*\address[1]{%
	\csgdef{address@\the\c@author}{#1}}
\AtEndDocument{%
	\xdef\author@count{\the\c@author}%
	\c@author=1
	\print@authors}
\newcommand*\print@authors{%
	\ifnum\c@author>\author@count
	\else
	\print@author{\the\c@author}%
	\advance\c@author by 1
	\expandafter\print@authors
	\fi}
\newcommand*\print@author[1]{%
	\par\medskip
	\begin{tabular}{@{}l@{}}%
		\textsc{}\\
		\csuse{address@#1}\\
		\textit{E-mail address}:
		\href{mailto:\csuse{email@#1}}{\csuse{email@#1}}
\end{tabular}}
\makeatother

\title{The title}

\author{L.\:Kryvonos}
\address{Department of Mathematics, \\
	Vanderbilt University, Nashville, TN 37240, USA}
\email{liudmyla.kryvonos@vanderbilt.edu}

\author{E.\:B.\:Saff}
\address{Center for Constructive Approximation, Department of Mathematics, \\
	Vanderbilt University, Nashville, TN 37240, USA}
\email{edward.b.saff@vanderbilt.edu}

\usepackage{graphicx}
\graphicspath{ {images/} }
\usepackage{enumitem} 
\usepackage{amssymb} 
\usepackage{bbold} 
\usepackage{subfig}
\usepackage{tikz-cd}

\usepackage{dsfont}
\usetikzlibrary{matrix}

\graphicspath{ {images/} }

\newcommand{\sm}{\textnormal{supp}}

\newcommand{\me}{\mu^{*}}
\newcommand{\men}{\nu^{*}}

\newcommand{\R}{\mathbb{R}}
\newcommand{\C}{\mathbb{C}}
\newcommand{\ra}{\rightarrow}
\newcommand{\ve}{\varepsilon}

\newtheorem{theorem}{Theorem}[section]
\newtheorem{lemma}[theorem]{Lemma}

\newtheorem{corollary}[theorem]{Corollary}

\theoremstyle{definition}
\newtheorem{definition}[theorem]{Definition}

\theoremstyle{remark}

\usepackage{enumitem}
\newlist{steps}{enumerate}{1}
\setlist[steps, 1]{label = Step \arabic*:}
\newcommand{\beq}{\begin{equation}}
	\newcommand{\eeq}{\end{equation}}
\makeatletter
\newcommand{\colim@}[2]{%
	\vtop{\m@th\ialign{##\cr
			\hfil$#1\operator@font colim$\hfil\cr
			\noalign{\nointerlineskip\kern1.5\ex@}#2\cr
			\noalign{\nointerlineskip\kern-\ex@}\cr}}%
}
\newcommand{\colim}{%
	\mathop{\mathpalette\colim@{\rightarrowfill@\scriptscriptstyle}}\nmlimits@
}
\renewcommand{\varprojlim}{%
	\mathop{\mathpalette\varlim@{\leftarrowfill@\scriptscriptstyle}}\nmlimits@
}
\renewcommand{\varinjlim}{%
	\mathop{\mathpalette\varlim@{\rightarrowfill@\scriptscriptstyle}}\nmlimits@
}
\makeatother

\begin{document}

	\title{On a problem of E. Meckes for the unitary eigenvalue process on an arc}
	
	\date{}
	\maketitle
	    
    \begin{abstract}
    	We study the \begingroup
    	\renewcommand{\thefootnote}{[\arabic{footnote}]}% Modify footnote globally
    	problem \footnote{The problem was communicated to the second author shortly before the untimely death of Professor Meckes.} \endgroup originally communicated by E. Meckes on the asymptotics for the eigenvalues of the kernel of the unitary eigenvalue process of a random $n \times n$ matrix. The eigenvalues $p_{j}$ of the kernel are, in turn, associated with the discrete prolate spheroidal wave functions. We consider the eigenvalue counting function  $|G(x,n)|:=\#\{j:p_j>Ce^{-x n}\}$, ($C>0$ here is a fixed constant) and establish the asymptotic behavior           
    	of its average over the interval $x \in (\lambda-\varepsilon, 
    	\lambda+\varepsilon)$ by relating the function $|G(x,n)|$ to the solution $J(q)$ of the following energy problem on the unit circle $S^{1}$, which is of independent interest. Namely, for given $\theta$, $0<\theta< 2 \pi$, and given $q$, $0<q<1$, we determine the function
    	$J(q) =\inf \{I(\mu): \mu \in \mathcal{P}(S^{1}), \mu(A_{\theta}) = q\}$, where $I(\mu):= \iint \log\frac{1}{|z - \zeta|} d\mu(z) d\mu(\zeta)$ is the logarithmic energy of a probability measure $\mu$ supported on the unit circle and $A_{\theta}$ is the arc from $e^{-i \theta/2}$ to $e^{i \theta/2}$. 
    \end{abstract}  
	
	
	
	\section{Introduction}
	For $n \in \mathbb{N}$, denote by $\mathbb{U}(n)$ the set of $n \times n$ unitary matrices over $\mathbb{C}$. Since the set $\mathbb{U}(n)$ forms a compact topological group with respect to matrix multiplication and the usual topology, there exists a unique probability measure on  $\mathbb{U}(n)$ (called Haar measure) that is invariant under left- and right-translation. 
	In other words, a distribution of a uniform random element $U_{n}$ of $\mathbb{U}(n)$ is the (Haar) measure $\mu$ on $\mathbb{U}(n)$ such that for any measurable subset $\mathcal{A} \subset \mathbb{U}(n)$ and any fixed matrix $M \in \mathbb{U}(n)$,
	$$
	\mu(M \mathcal{A}) = \mu(\mathcal{A} M) = \mu(\mathcal{A}),
	$$
	where $M \mathcal{A}:=\{MU: \;\;U \in \mathcal{A}\}$ and $ \mathcal{A}M:=\{UM: \;\;U \in  \mathcal{A}\}$.
	
	 Denote the eigenvalues of a Haar-distributed random unitary $n\times n$ matrix $U_n$, all of which lie on the circle $\mathbb{S}^{1}:=\{z \in \mathbb{C}: |z| =1\}$, by $\{e^{i\theta_1},\dots,e^{i\theta_n}\}$, $0 \leq \theta < 2 \pi$. Our main goal is to study the eigenvalue counting function
	\beq \label{N}
	\mathcal{N}_{\theta} =\mathcal{N}_{(0,\theta)} :=\#\{j:0<\theta_j<\theta\},
	\eeq
	where $\theta\in(0,2\pi)$ is fixed, by utilizing the determinantal structure of the eigenvalue process for $\mathcal{N}_{\theta}$ that we describe next.
	
     A \textit{point process} $\mathcal{X}$ in a locally compact Polish space $X$ is a random discrete subset of $X$. Let $\mathcal{N}_{A}$ denote the (random) number of points of $\mathcal{X}$ in $A \subset X$. We say that $\mathcal{X}$ is a \textit{determinantal point process} if for any finite number of pairwise disjoint subsets $A_{1},...,A_{k} \subset X$,
    $$
    	\mathbb{E} \bigg[\prod_{j=1}^{k} \mathcal{N}_{A_{j}}\bigg] = \int_{A_{1}} ... \int_{A_{k}}  \det[K(x_{i},x_{j})]_{i,j=1}^{k} d\mu(x_{1})... d\mu(x_{k}),
    $$
    for some kernel $K: X \times X \rightarrow [0,1]$ and Borel measure $\mu$ on $X$.
    One of the remarkable properties of eigenvalue distributions of matrices from the compact classical groups is that they are determinantal point processes.
     The connection in the case of the unitary group has been known at least since the work \cite{Dyson} of F. J. Dyson, while for other groups one of the earliest references known to us is \cite{KS} by N. M. Katz and P. Sarnak. We state Theorem \ref{DPP} below for the unitary group only, and refer the reader to \cite{HaarBook} for more general results concerning other groups.

	\begin{theorem} (see \cite[Proposition 3.7]{HaarBook}) \label{DPP}
		The eigenvalue angles of uniformly distributed random matrices in $\mathbb{U}(n)$ is a determinantal point process with respect to uniform measure on $[0, 2\pi)$, with kernel
		\beq\label{kernel}
		K_{n}(x,y) := \sum_{j=0}^{n-1} e^{ij (x-y)}.
		\eeq
		\end{theorem}
	For convenience, we will use an alternative form of the kernel (\ref{kernel}):
	\beq \label{kernel2}
		K_{n}(x,y) : = \begin{cases}
			\sin(\frac{n(x-y)}{2})/\sin(\frac{(x-y)}{2}), \;\;\;\;\;\; \textnormal{if} \;\; x-y \ne 0,\\
			n, \hspace{1.7in} \textnormal{if} \;\; x-y = 0.\\
		\end{cases}
	\eeq
     Although the kernels (\ref{kernel}) and (\ref{kernel2}) are different functions, they are unitarily similar and thus define the same point processes (see \cite[Section 5.4]{KS}).
     
   Theorem \ref{DPP} allows us to apply the next result.
	  
	 \begin{theorem} (J. B. Hough et al., \cite{Hough}, see also \cite[Theorem 4.1]{HaarBook}) \label{operator}
	 	
	 	Let $K: X \times X \rightarrow \C$ be a kernel on a locally compact Polish space $X$ and $\mu$ a Borel measure on $X$ such that the corresponding integral operator $\mathcal{K}: L^{2} (X,\mu) \rightarrow L^{2}(X,\mu)$ defined by
	 	$$
	 	\mathcal{K}(f)(x) := \int K(x,y) f(y)d\mu(y)
	 	$$
	 	is self-adjoint, nonnegative, and locally trace-class with eigenvalues in [0,1]. For $D \subset X$ measurable, let $K_{D}(x,y) := \mathbb{I}_{D}(x) K(x,y)\mathbb{I}_{D}(y)$ be the restriction of $K$ to $D \times D$. Suppose that $D$ is such that the operator $\mathcal{K}_{D}$ with kernel $K_{D}$ is trace-class. Denote by $\{p_{j}\}_{j \in \mathcal{J}}$ the eigenvalues of the operator $\mathcal{K}_{D}$ on $L^{2}(D, \mu)$ (the index set $\mathcal{J}$ may be finite or countable) and denote by $\mathcal{N}_{D}$ the number of particles of the determinantal point process with kernel $K$ that lie in $D$. Then
	 	$$
	 	\mathcal{N}_{D} \overset{d}{=} \sum_{j \in \mathcal{J}} \xi _{j},
	 	$$
	 	where $``\overset{d}{=}"$ denotes equality in distribution and the $\xi_{j}$ are independent Bernoulli random variables with $\mathbb{P}[\xi_{j} = 1] = p_{j}$ and $\mathbb{P}[\xi_{j} =0] = 1 - p_{j}$.
	 	\end{theorem}
 	 
 	     Let $X=[0, 2\pi)$ and $\mu$ be a uniform measure on $X$. Consider an integral operator $\mathcal{K}_{n}$ with kernel $K_{n}$ from (\ref{kernel2}), restricted to $[0, \theta] \times [0, \theta]$. It's easy to see that the operator satisfies all conditions of Theorem \ref{operator}. Moreover, since its kernel is degenerate, $\mathcal{K}_{n}$ has $n$ eigenvalues $\{p_{j}\}_{j=1}^{n}$ (see, for example, \cite[Section 3.2]{Operator}); hence, by Theorem \ref{operator}, the counting function $\mathcal{N}_\theta$, introduced in (\ref{N}), is equal in distribution to a sum of $n$ independent Bernoulli random variables: 
 	     \beq \label{ND}
 	     \mathcal{N}_{\theta}\stackrel{d}{=}\sum_{j=1}^n\xi_j,
 	     \eeq
 	     where $\mathbb{P}[\xi_j=1]=p_j$ and $\mathbb{P}[\xi_j=0]=1-p_j$. We would like to study the behavior of the $p_j$ near zero. Namely, for $C>0$ large and fixed and $x \in \R$, consider 
 	     \beq \label{Gfunct}
 	     G(x,n):=\{j:p_j>Ce^{-x n}\},
 	     \eeq
 	     and let $|G(x,n)|$ be the number of $p_{j}$ in $G(x,n)$.
 	      Our main result is the following theorem that describes the asymptotic behavior of the function $|G(x,n)|$.
 	     \begin{theorem} \label{mainTh}
 	     	Let $\theta \in (0, 2\pi)$. For any fixed $\ve >0$ and $\lambda \in \mathbb{R}$,
 	     	\beq \label{mainform}
 	     	 \frac{1}{2\ve}\int_{\lambda -\ve}^{\lambda+\ve} |G(x,n)|dx = \frac{n}{2 \ve} (\Lambda(\lambda + \ve) - \Lambda(\lambda - \ve)) - o(n), \;\; n \rightarrow \infty,
 	     	\eeq
 	     	where $o(n)$ depends on $\lambda$, $\ve$ and the function $\Lambda(\lambda) := \underset{y \in [0,1]}{\sup}\{\lambda y - J(y)\}$ is the Fenchel-Legendre transform of the rate function $J(y) = \iint \log \frac{1}{|z - \zeta|} d \men_{y}(z) d \men_{y}(\zeta)$ with $\men_{y}$ given by (\ref{densityCircle}) of Theorem \ref{Sol I} below (with $q$ replaced by $y$).
 	     	
 	     	Furthermore, there exists a constant $c_{0}>0$, such that for all $\lambda \geq c_{0}$ expression (\ref{mainform}) becomes
 	     	$$ \frac{1}{2\ve}\int_{\lambda -\ve}^{\lambda+\ve} |G(x,n)|dx = n - o(n), \;\; n \rightarrow \infty. 
 	     	$$
 	     \end{theorem} 
      \noindent
      	\textbf{Remark}. The above functions $G, \Lambda$ and $J$ depend on the constant $\theta$, though for simplicity we omit this dependence in our notations.\\

 	     To prove Theorem \ref{mainTh}, we use a key observation of T. Liu and E. Meckes \cite{Meckes} that the random variables $\frac{\mathcal{N}_{\theta}}{n}$ satisfy the large deviation principle, with the rate function $J(x)$ being a solution to some equilibrium problem on the unit circle. The function $\widehat{J}(x)$ that solves an equivalent problem on the interval $[-1,1]$, in turn, is a limit distribution of zeros of Heine–Stieltjes polynomials $-$ the polynomial solutions to the generalized Lam\'{e} differential equation, and was obtained in the work \cite{Andrei} by A. Mart\'{\i}nez-Finkelshtein and E. B. Saff. In Section \ref{ensect} we give an alternative proof of this result as well as establish the following theorem.
 	      
 	      \begin{theorem} \label{Sol I} $ $
 	      	
 	      	The measure $\men=:\men_{q} \in \mathcal{P}(\mathbb{S}^{1})$, such that 
 	      	$I(\men) =\inf \{I(\nu): \nu \in \mathcal{P}(\mathbb{S}^{1}), \nu(A_{\theta}) = q\}$, where $I(\nu):= \iint \log\frac{1}{|z - \zeta|} d\nu(z) d\nu(\zeta)$ and $A_{\theta} := \{z \in \mathbb{S}^{1}: -\frac{\theta}{2}\leq \arg z \leq \frac{\theta}{2}\}$, is unique and

 	      	i) if $q \geq \frac{\theta}{2 \pi}$, given by
 	      	\beq \label{densityCircle}
 	      	d\men(e^{i\psi}) = \frac{\sqrt{|\cos(\psi) - \alpha|}}{2 \pi\sqrt{|\cos(\psi) - \cos(\theta/2)|}} d\psi, \;\;\; 
 	      	\eeq
 	      	where
 	      	$
 	      	e^{i \psi} \in A_{\theta} \cup \{z \in \mathbb{S}^{1}: \arccos(\alpha) \leq \arg \; z \leq 2\pi - \arccos(\alpha)\},
 	      	$
 	      	with $\alpha$ determined from the equation
 	      	\beq  \label{alpha2}
 	      	\int_{-1}^{\alpha} \frac{\sqrt{|x - \alpha|}}{\pi \sqrt{|(x+1) (x-\cos(\theta/2)) (x-1)|}} dx = 1 - q;
 	      	\eeq
 	      	
 	      	ii) if $q \leq \frac{\theta}{2 \pi}$, given by (\ref{densityCircle}), where $
 	      	e^{i \psi} \in A^{c}_{\theta} \cup \{z \in \mathbb{S}^{1}: - \arccos(\alpha) \leq \arg \; z \leq \arccos(\alpha)\}
 	      	$
 	      	and $\alpha$ is a solution to the equation
 	      	$$
 	      	\int_{-1}^{\cos(\theta/2)} \frac{\sqrt{|x - \alpha|}}{\pi \sqrt{|(x+1) (x-\cos(\theta/2)) (x-1)|}} dx = 1 - q.
 	      	$$
 	      \end{theorem}
 	     
 	      	\vspace{\baselineskip}

 	    Energy minimization problems, similar to the one we encounter in the work, appear naturally when studying statistical systems of many particles in the framework of the so-called log gases, where the particles are treated as a system of point charges on one or two dimensional sets, subject to the logarithmic interaction (see, for example, \cite{Ameur}, \cite{Chafai}). A similar energy problem appears in the work \cite{Charlier1}, \cite{Charlier2}, in the context of study of the determinants of Toeplitz-type operators. In particular, as an application of the results obtained in \cite{Charlier1}, the authors obtain an estimate for the probability that a random unitary matrix has all its eigenvalues concentrated on an arc of the unit circle (see also \cite{Widom} in connection to this result).
 	     We would like to note that \cite{Charlier1}, \cite{Charlier2} deal with the energy problem with an external field, while Theorem \ref{Sol I} concerns minimization of the logarithmic energy under the constraints on masses of the measure.

 	      Note that, with a suitable change of variables, the operator with kernel (\ref{kernel2}) can be written as
        \beq \label{opA}
        \mathcal{K}_{n}(f)(x) = \int_{- W}^{ W}  \frac{\sin n\pi (x-y)}{\sin \pi (x-y)} f(y)dy, \;\; x \in [-W, W],
        \eeq
        with $W = \frac{\theta}{4 \pi}$. An alternative way to study the behavior of its eigenvalues $p_{j}$ is from the operator theory perspective, and below we give a short summary of known results in this direction.

        Operators of the form (\ref{opA}) with $W \in (0, 1/2)$, whose eigenfunctions are well-known discrete prolate spheroidal wave functions, were extensively studied by D. Slepian in \cite{Slepian}. Along with establishing the asymptotics for the   eigenfunctions of the operator, he also examined the behavior of its eigenvalues. In particular, he showed in \cite{Slepian} that the operator (\ref{opA}) has $n$ distinct eigenvalues and the first $2nW$ of them (arranged in descending order) tend to cluster extremely close to 1, while the remaining eigenvalues tend to cluster similarly close to 0. These results, though, do not address precisely how many eigenvalues one can expect to find between $\ve$ and $1-\ve$. Numerical experiments performed by S. Karnik, J. Romberg, and M. A. Davenport in \cite{Karnik} suggest that the number of eigenvalues $p_{j} = p_{j}(W,n)$ in the ``transition region" between $\ve$ and $1-\ve$ behaves like $\#\{j : \ve<p_{j}<1-\ve\} = O(\log(nW) \log\big(\frac{1}{\ve}\big))$. They obtained the nonasymptotic bounds for the number of eigenvalues in $(\ve, 1-\ve)$ that captures the observed logarithmic dependence of the transition region on $n$ and $\ve$, and compared them with the previous known nonasymptotic bounds (see \cite{Boulsane}, \cite{Karnik2}, \cite{Zhu}).
        
         The paper is organized as follows: in Section \ref{Sect2} we prove Theorem \ref{mainTh},
           in Section \ref{ensect} we solve the constrained energy problem on an interval and the unit circle to find the rate function $J(x)$ from Theorem \ref{mainTh}.

	\section{Asymptotics for $p_{j}$.} \label{Sect2}
	 
	   To study the function $|G(x,n)|$, given by (\ref{Gfunct}), we make use of the known large deviation principle (LDP) due to F. Hiai and D. Petz \cite{Hiai} for the empirical spectral measure $\mu_n$ of $U_n$.
	
	\begin{definition}
	A sequence of probability Borel measures $\{P_{n}\}$ on a topological space $X$ satisfies an LDP with rate function $I$ and speed $s_{n}$ if for all Borel sets $\mathcal{B} \subseteq X$,
	$$
	- \underset{x \in \mathcal{B}^{0}}{\inf} I(x) \leq \underset{n \rightarrow \infty}{\lim \inf} \frac{1}{s_{n}} \log(P_{n}(\mathcal{B}))  \leq \underset{n \rightarrow \infty}{\lim \sup} \frac{1}{s_{n}} \log(P_{n}(\mathcal{B})) \leq 	- \underset{x \in \overline{\mathcal{B}}}{\inf} I(x)
	$$
	\end{definition}

	\begin{theorem}{(F. Hiai, D. Petz, \cite[Section 1]{Hiai})}

	 Let $U_{n} \in \mathbb{U}(n)$ and $\mu_{n} : = \frac{1}{n} \sum _{j=1}^{n} \delta(e^{i \theta_{j}})$, where $\{e^{i \theta_{j}}\}_{j=1}^{n}$ are the eigenvalues of $U_{n}$ and $\delta(\zeta)$ denotes the Dirac measure at $\zeta$.
	 Denote by $P_n$ the law of $\mu_n$. Then the sequence $\{P_n\}$ satisfies an LDP on the space $\mathcal{P}(\mathbb{S}^1)$ of probability measures on the unit circle equipped with the topology of weak convergence, with speed $n^2$ and strictly convex rate function
	\[I(\nu)=-\iint\displaylimits_{\mathbb{S}^{1} \times \mathbb{S}^{1}}\log|z-w|d\nu(z)d\nu(w).\]
\end{theorem}

	Let $A_{\theta}$ denote the arc of the circle consisting of those points with argument in $(0,\theta)$. Then $\nu\mapsto \nu(A_\theta)$ is a continuous function of $\nu$ (with respect to the topology of weak convergence), and so by the contraction principle \cite[Section 4.2]{Dembo}, the random variables $\mu_n(A_\theta)=\frac{\mathcal{N}_\theta}{n}$ satisfies an $LDP$ on $[0,1]$ with speed $n^{2}$ and rate function 
	\beq \label{min}
	J(y)=\inf\{I(\nu):\nu\in\mathcal{P}(\mathbb{S}^1),\nu(A_\theta)=y\}.
	\eeq
	
	 In Section \ref{ensect} we solve the minimization problem (\ref{min}) and give an explicit formula for the function $J(y)$. This allows us to utilize the next result and investigate the behavior of moment-generating functions of the sequence of random variables $\{\frac{\mathcal{N}_\theta}{n}\}$.
	\begin{theorem}{(Varadhan's Integral Lemma, \cite[Section 4.3]{Dembo})} \label{Var}
		
		Suppose that the family of random variables $\{Z_{\varepsilon}\}$ taking values in a topological space $X$ satisfies the LDP with rate function $J: X \rightarrow [0, \infty]$. Let $\phi: X \rightarrow \mathbb{R}$ be any continuous function. Assume further the following moment condition for some $\gamma>1$,
		\beq \label{cond2}
		\underset{\ve \ra 0}{\lim \sup}\;\ve \log \mathbb{E} [e^{\gamma \phi(Z_{\ve}) /\ve} ] < \infty.
		\eeq
		Then 
		\beq \label{FL}
		\underset{\ve \ra 0}{\lim}\; \ve \log \mathbb{E} [e^{ \phi(Z_{\ve}) /\ve} ] = \underset{y \in X}{\sup} \{\phi(y) - J(y)\}.
		\eeq
	\end{theorem}
	Therefore, if we take above $\phi(y) = \lambda y, \; \lambda \in \mathbb{R}$
	and show that condition (\ref{cond2}) holds with $\varepsilon =1/n$ and $Z_{\varepsilon} = \mathcal{N}_\theta$, then Theorem \ref{Var} implies that the next limit exist
	\begin{align} \label{lambda}
		\Lambda(\lambda):=\lim_{n\to\infty}\frac{1}{n^2} \log \mathbb{E}[e^{\lambda n\mathcal{N}_\theta}].
	\end{align}
It follows from (\ref{FL}) that the function $\Lambda(\lambda)$ appearing in (\ref{mainform}) is given by the Fenchel-Legendre transform of $J$, i.e., $\Lambda(\lambda) = \underset{y \in [0,1]}{\sup} \{\lambda y - J(y)\}$. On the other hand, in the proof of Theorem \ref{mainTh} we express the right side of (\ref{lambda}) in terms of the function $|G(x,n)|$ and from there obtain the relationship between $\Lambda(\lambda)$ and $|G(x,n)|$.
	\subsection{Proof of Theorem \ref{mainTh}}
  \begin{proof}
	As mentioned above, we are going to examine the limit in equation (\ref{lambda}). Note first that if $\lambda<0$, then $\lim_{n\to\infty}\frac{1}{n^2} \log \mathbb{E}[e^{\lambda n\mathcal{N}_\theta}]=0$ trivially, so we restrict our attention to $\lambda>0$. Since $\mathcal{N}_\theta\stackrel{d}{=}\sum_{j=1}^n\xi_j$ with the $\xi_j$ independent Bernoullis with success probabilities $p_j$ that are arranged in decreasing order, 
	\begin{align*}
		\frac{1}{n^2} \log \mathbb{E}[e^{\lambda n\mathcal{N}_\theta}]=\frac{1}{n^2}\sum_{j=1}^n\log(p_je^{\lambda n}+1-p_j)\\
		=\frac{1}{n^2}\sum_{j\in G(\lambda,n)}\log(p_je^{\lambda n}+1-p_j)+\frac{1}{n^2}\sum_{j\notin G(\lambda,n)}\log(p_je^{\lambda n}+1-p_j)\\
		=:S_1+S_2
	\end{align*}
	For $j\notin G(\lambda,n)$ and any $\lambda>0$, $p_j\leq p_je^{\lambda n}\leq C$, and so
	\[0\leq\sum_{j\notin G(\lambda,n)}\log(p_je^{\lambda n}+1-p_j)\leq n\log(C).\]
	We thus have that $0\leq S_2\leq \frac{\log(C)}{n}$.\par
	For $S_1$, first observe that 
	\[\log(p_je^{\lambda n}+1-p_j)=\lambda n+\log(p_j)+\log(1+\frac{1-p_j}{p_j}e^{-\lambda n})\]
	and that, for $j\in G(\lambda,n)$, $\frac{1-p_j}{p_j}e^{-\lambda n}<\frac{1}{C}$, so that
	\begin{align}
		S_1=\frac{\lambda|G(\lambda,n)|}{n}+\frac{1}{n^2}\sum_{j\in G(\lambda,n)}\log(p_j)+O(\frac{1}{n}),
	\end{align}
	where the implied constant depends only on $C$.\par
	For $j\in G(\lambda,n)$, define $\lambda_j$ such that $p_j=Ce^{-\lambda_j n}$ and set $\lambda_{0}:=0$ (note that $\lambda_{j} > 0$, all $j \geq 1$, since $G(\lambda,n)=\emptyset$ for $\lambda\leq 0$). Then
	\begin{align*}
		\frac{1}{n^2}\sum_{j\in G(\lambda,n)}\log(p_j)=\frac{\log(C)|G(\lambda,n)|}{n^2}-\frac{1}{n}\sum_{j\in G(\lambda,n)}\lambda_j\\
		=-\frac{1}{n}\sum_{j=1}^M\lambda_j+O(\frac{1}{n}),
	\end{align*}
	where $M=\max\{j:j\in G(\lambda,n)\}$.\par
	Define
	\[g_{n}(\lambda):=\frac{|G(\lambda,n)|}{n}.\]
	Then, since the $\lambda_j$ are in increasing order, $g_{n}(\lambda_{j+1})-g_{n}(\lambda_j)=\frac{1}{n}$ for each $j$ and $g_{n}(\lambda_1) =0$, by summation by parts,
	\begin{align*}
		-\frac{1}{n}\sum_{j=1}^M\lambda_j=  - \frac{\lambda_{M}}{n} -\sum_{j=1}^{M-1}\lambda_j(g_{n}(\lambda_{j+1})-g_{n}(\lambda_j))\\
		= - \frac{\lambda_{M}}{n} + \lambda_1g_{n}(\lambda_1) - \lambda_{M-1}g_{n}(\lambda_M) + \sum_{j=2}^{M-1}g_{n}(\lambda_j)(\lambda_j-\lambda_{j-1})\\
		=  -\frac{\lambda_{M}}{n}  -\lambda_Mg_{n}(\lambda_{M}) + \sum_{j=1}^{M}g_{n}(\lambda_j)(\lambda_j-\lambda_{j-1}).
	\end{align*}
	Since
	\[\sum_{j=1}^Mg_{n}(\lambda_j)(\lambda_j-\lambda_{j-1}) = \int_0^{\lambda_M}g_{n}(x)dx,\]
	we thus have
	\begin{align*}
		S_1=\lambda g_{n}(\lambda)-\lambda_Mg_{n}(\lambda_M)+\int_0^{\lambda_M}g_{n}(x)dx+O(\frac{1}{n})\\
		=\lambda_M(g_{n}(\lambda) - g_{n}(\lambda_M)) + \int_0^{\lambda}g_{n}(x)dx +O(\frac{1}{n})\\
		= \int_0^{\lambda}g_{n}(x)dx +O(\frac{1}{n}),
	\end{align*}
that verifies (\ref{cond2}).
	As $n\to \infty$, by Theorem \ref{Var}, the limit
	
	 \beq  \label{limg}
	 \lim_{n\to\infty}\frac{1}{n^2} \log\mathbb{E}[e^{\lambda n\mathcal{N}_\theta}]= \lim_{n\to\infty} \int_0^\lambda g_{n}(x)dx = \Lambda(\lambda) 
	 \eeq  exists for all $\lambda$.
	 Notice that since we don't have any additional information about the behavior of $g_{n}(x)$ for each fixed $x \in[0,\lambda]$, we can't interchange the integration with a limit in (\ref{limg}).
	 
	  For a fixed $\ve >0$, (\ref{limg}) implies 
	  $$ \int_{\lambda -\ve}^{\lambda+\ve} g_{n}(x)dx \underset{n \ra \infty}{\longrightarrow} \Lambda(\lambda + \ve) - \Lambda(\lambda - \ve), $$
	  therefore we immediately get
	  
	  \beq \label{average}
	   \frac{1}{2\ve}\int_{\lambda -\ve}^{\lambda+\ve} |G(x,n)|dx = \frac{n}{2 \ve} (\Lambda(\lambda + \ve) - \Lambda(\lambda - \ve)) - o(n). 
	  \eeq
	   	
	  The function $\Lambda(\lambda)$, according to Theorem \ref{Var}, is given by 
	   $$\Lambda(\lambda) = \underset{y \in [0,1]}{\sup} \{\lambda y - J(y)\}$$
	   with $J(y) = \iint \log \frac{1}{|z - \zeta|} d \men_{y}(z) d \men_{y}(\zeta)$, where $\men_{y}$ is a measure from Theorem \ref{Sol I}. It is easy to check that  the left-hand derivative of $J(x)$ at $x=1$ exists, therefore, for all $\lambda \gg J'(1)$ and $\varepsilon$ sufficiently small, (\ref{average}) becomes
	   $$
	   \frac{1}{2\ve}\int_{\lambda -\ve}^{\lambda+\ve} |G(x,n)|dx = n - o(n) 
	   $$
	  that completes the proof.
	  \end{proof}
	  
	 \vspace{0.1in}

	\section{Solution to Energy problem}\label{ensect}

	Let $\Sigma \subset \mathbb{C}$ be a compact subset of the complex plane and $\mathcal{P}(\Sigma)$ the collection of all positive unit Borel measures supported on $\Sigma$. For a given $0<\theta<2\pi$ we use the notations   
	$A_{\theta} := \{z \in \mathbb{S}^{1}: -\frac{\theta}{2}\leq \arg z \leq \frac{\theta}{2}\}$, $A^{c}_{\theta} := \{z \in \mathbb{S}^{1}: \frac{\theta}{2}\leq \arg z \leq 2\pi - \frac{\theta}{2}\}$ for the corresponding subarcs of the unit circle.

	The \textit{logarithmic energy} $I(\mu)$ and the \textit{logarithmic potential} $U^{\mu}$ of a measure $\mu \in \mathcal{P}(\Sigma)$ are defined, respectively, as
	\beq \label{energy}
	I(\mu):= \iint \log\frac{1}{|z - \zeta|} d\mu(z) d\mu(\zeta),
	\eeq
	
	\beq \label{potential}
	U^{\mu}(z):= \int \log\frac{1}{|z - \zeta|}  d\mu(\zeta).
	\eeq
	
	In this section we study the following two related problems:
	
	\vspace{\baselineskip}
	
	\textbf{Problem I.}  Given $q$ and $\beta$, with $0<q<1$, $-1<\beta< 1$, determine a measure $\mu \in \mathcal{P}([-1,1])$ that minimizes the energy  $I(\mu)$, subject to the constraint $\mu([\beta, 1]) = q$.
	
	\vspace{\baselineskip}
	
	\textbf{Problem II.}   
	Given $q$ and $\theta$, with $0<q<1$, $0 < \theta < 2\pi$, determine a measure $\nu \in \mathcal{P}(\mathbb{S}^{1})$ that minimizes the energy  $I(\nu)$, subject to the constraint $\nu(A_{\theta}) = q$.
	
	\vspace{\baselineskip}

	The limiting version of the Problems I and II with constraints on the mass at fixed points, $\beta =1$ and $\theta = 0$, respectively, were already treated in a couple of works. The description of the equilibrium charge distribution of amount $1-q$ on the unit circle when a fixed charge amount $q >0$ is placed at $t=1$ was done in \cite{LSV} by M. Lachance, E. B. Saff and R.S. Varga.  
	A similar problem on an interval was studied in \cite{SUV} by E. B. Saff, J.L. Ullman and R.S. Varga, where the the equilibrium charge distribution of amount $1-q_{1}-q_{2}$ was determined when a charge amount $q_{1} >0$ is placed at $t=1$, and a charge amount $q_{2} >0$ is placed at $t=-1$, $q_{1}+ q_{2} <1$. Thus, the solution to the Problems I and II, given by Theorems \ref{Sol II} and \ref{Sol I} below, can be considered as an extension of the results obtained in \cite{LSV} and \cite{SUV} (with $q_{2} =0$), respectively. We show that in the limiting case, when $\beta \rightarrow 1$ in Problem I, our result recovers the one from \cite{SUV} .
	
	\begin{theorem}{(A. Martínez-Finkelshtein, E.B. Saff, \cite[Section 4]{Andrei})} \label{Sol II} $ $
		
		The measure $\me=:\me_{q} \in \mathcal{P}([-1,1])$ such that 
		$I(\me) =\inf \{I(\mu): \mu \in \mathcal{P}([-1,1]), \mu([\beta, 1]) = q\}$ is unique and
		
		i) if $q \geq \frac{1}{\pi} \int_{\beta}^{1} \frac{1}{\sqrt{1 - x^{2}}} dx$, is given by 
		\beq \label{densitySol}
		d\me(x) = \frac{\sqrt{|x - \alpha|}}{\pi \sqrt{|(x+1)(x-\beta)(x-1)|}} dx, \;\;\;
		\eeq
		where $ x \in [-1, \alpha] \cup [\beta, 1]$, $\alpha \leq \beta$, and with $\alpha$ determined from the equation
		\beq  \label{alpha}
		\int_{-1}^{\alpha} \frac{\sqrt{|x - \alpha|}}{\pi \sqrt{|(x+1) (x-\beta) (x-1)|}} dx = 1 - q;
		\eeq 
		
		ii) if $q \leq \frac{1}{\pi} \int_{\beta}^{1} \frac{1}{\sqrt{1 - x^{2}}} dx$, is given by (\ref{densitySol}), where $x \in [-1, \beta] \cup [\alpha, 1]$, $\alpha \geq \beta$, and $\alpha$ is the solution to the equation
		$$
		\int_{-1}^{\beta} \frac{\sqrt{|x - \alpha|}}{\pi \sqrt{|(x+1) (x-\beta) (x-1)|}} dx = 1- q.
		$$
	\end{theorem}
	
	Note that in the trivial case, when $q = \frac{1}{\pi} \int_{\beta}^{1} \frac{1}{\sqrt{1 - x^{2}}} dx$, we get $\alpha=\beta$ and so Theorem \ref{Sol II} gives, as expected, the arcsine distribution, $d \me(x) = \frac{1}{\pi \sqrt{1 - x^2}}$.

	\begin{corollary}
		When $\beta \rightarrow 1$, the density function in (\ref{densitySol}) becomes
		$$
		d\me(x) = \frac{\sqrt{|x - \alpha|}}{\pi \sqrt{(x+1)}(1-x)} dx, \;\;\; x \in [-1, \alpha], 
		$$
		where $\alpha$ is the solution to the equation
		\beq 
		\int_{-1}^{\alpha} \frac{\sqrt{|x - \alpha|}}{\pi \sqrt{(x+1)}(1-x)} dx = 1 - q.
		\eeq
		This is precisely the measure obtained in \cite{SUV} that corresponds to the case $q_{1} =q$, $q_{2} = 0$.
	\end{corollary}

	Once Theorem \ref{Sol II} is verified, the solution to the Problem II is immediate, as we show in Section \ref{Proof2}.

	\subsection{Auxiliary results} \label{aux.res.}

	We call $\mu \in \mathcal{P}([a,b])$ the \textit{unconstrained} equilibrium measure on an interval $[a,b]$ if $\mu $ minimizes the logarithmic energy $I(\mu)$ among all unit measures in $\mathcal{P}([a,b])$. It is well-known (see e.g. \cite[Section I.3]{BookST}) that
	$$
	d\mu(x) = \frac{1}{\pi}\frac{dx}{ \sqrt{(x-a)(b-x)}}, \;\;\; x \in [a,b].
	$$
	
	The measure $\me$ that minimizes the energy in Problem I can be viewed as the measure of the form 
	$$
	\me = (1-q) \me_{1} + q \me_{2},
	$$
	where $\me_{1} \in \mathcal{P}([-1,\beta])$ and $\me_{2} \in \mathcal{P}([\beta, 1])$. In order to characterize the measure $\me$, we refer to \cite[Theorem VIII.2.1]{BookST}, which states that a measure $\mu$ is globally optimal if and only if each component of $\mu$ is optimal when the other one is kept fixed and regarded as external field.
	Although the theorem is originally formulated for the case of closed sets of positive distance from one another, it is also valid in our setting, when two closed intervals have disjoint interiors.
	According to \cite[Theorem VIII.2.1]{BookST}, there exist constants $F_{1}$, $F_{2}$, such that the potential of the measure $\me$ satisfies
	\beq \label{Frostman m1}
	U^{\me}(z) \geq F_{1}, \;\;  \textnormal{q.e. on} \; [-1, \beta], \;\;\;\;\;  U^{\me}(z) = F_{1}, \;\; \textnormal{q.e. on} \; \sm \:\me_{1}, 
	\eeq
	
	\beq \label{Frostman m2}
	U^{\me}(z) \geq F_{2}, \;\;  \textnormal{q.e. on} \; [\beta, 1], \;\;\;\;\;  U^{\me}(z) = F_{2}, \;\; \textnormal{q.e. on} \;\sm\: \me_{2}.
	\eeq
	
	We remark that a similar reasoning can be applied to characterize the optimal measure $\men \in \mathbb{S}^{1}$ in Problem II.
	
	Our next goal is to show that the measure $\me$ is absolutely continuous with respect to the  measure $\sigma_{1}+ \sigma_{2}$, where $\sigma_{1}$, $\sigma_{2}$ are (unconstrained) equilibrium measures on $[-1,\beta]$ and $[\beta,1]$, respectively, using the following result of de La Vall\'{e}e Poussin.
	\begin{theorem}(see \cite[Theorem IV.4.5.]{BookST}) \label{Valle-Poussin}
		
		Let $\mu$ and $\nu$ be two measures of compact support, and let $\Omega$ be a domain in which both $U^{\mu}$, $U^{\nu}$ are finite and satisfy with some constant $c$ the inequality
		\beq \label{potIneq}
		U^{\mu}(z) \leq U^{\nu}(z) + c, \;\;\; z \in \Omega.
		\eeq
		If $A$ is the subset of $\Omega$ in which equality holds in (\ref{potIneq}), then $\nu \vert_{A} \leq \mu \vert_{A}$; that is, for every Borel set $B \subset A$, $\nu(B) \leq \mu(B)$.
	\end{theorem}
	
	\begin{lemma}\label{AbsCont}
		The measures $\me_{1}$ and $\me_{2}$ that solve Problem I are absolutely continuous with respect to the equilibrium measure $\sigma_{1}$ on $[-1,\beta]$ and, respectively, equilibrium measure $\sigma_{2}$ on $[\beta,1]$.
		
	\end{lemma}
	\begin{proof}
		Let $\sigma_{1} \in \mathcal{P}([-1,\beta])$ be the (unconstrained) equilibrium measure on $[-1,\beta]$. Then $\sigma_{1}$ satisfies	
		\beq \label{Frostman sg}
		U^{\sigma_{1}} = C, \;\;   \textnormal{ on} \; [-1,\beta],
		\eeq
		where $C$ is an explicitly known constant.
		Combining (\ref{Frostman sg}) with (\ref{Frostman m1}), we obtain
		$$
		U^{\sigma_{1}}  = 	U^{\me} + C - F_{1}, \;\;\;  \textnormal{q.e. on} \; \sm \:\me_{1},
		$$
		\beq \label{Dominat1}
		U^{\sigma_{1}} \leq 	U^{\me} + C - F_{1}, \;\;\;  \textnormal{q.e. on} \; [-1, \beta]. 
		\eeq

		Since $\|\sigma_{1}\| = \|\me\| =1$, the principle of domination for logarithmic potentials (see, e.g. \cite[Theorem II.3.2]{BookST}) implies that the inequality (\ref{Dominat1}) holds for every $z \in \mathbb{C}$. By Theorem \ref{Valle-Poussin}, we have $(1-q)\me_{1} \leq \sigma_{1}$. Similarly, we can use the (unconstrained) equilibrium measure $\sigma_{2} \in \mathcal{P}([\beta,1])$ on $[\beta, 1]$ to show that $q\me_{2} \leq \sigma_{2}$.
	\end{proof}

	\subsection{Proof of Theorem \ref{Sol II}}\label{Proof1}

	\begin{proof} 
		\textbf{Step 1.} Determining the support of $\me$. 
		
		Let $\sigma \in \mathcal{P}([-1,1])$ denote the (unconstrained) equilibrium measure on $[-1,1]$ and let $m_{1}:= \sigma([-1,\beta])$, $m_{2}:=\sigma([\beta,1])$. If $q = m_{2}$, then  it's clear that $\me = \sigma$ and $\sm \; \me = [-1,1]$.
		
		We will now assume that $q > m_{2}$ (the proof for the case $q < m_{2}$ is analogous). We shall show that in this case the support of the minimizer is a union of two intervals, $[-1, \alpha_{0}] \cup [\beta_{0}, 1]$, $\alpha_{0} <  \beta_{0}$. For this purpose, we denote as before the restriction of the measure $\me$ to $[-1, \beta]$ and $[\beta, 1]$ by $(1-q)\me_{1}$ and $q\me_{2}$, respectively.
		First, notice that $U^{\me_{2}}(z)$ is a convex function on $[-1, \beta)$, which follows immediately from 
		
		$$
		\frac{d^{2}U^{\me_{2}}(z)}{dz^{2}} = \int \frac{d \me_{2}(\zeta)}{(z-\zeta)^{2}} \geq 0, \;\;\;\; z \in [-1, \beta).
		$$
		Since the measure $\me_{1}$ is the solution to the equilibrium problem on $[-1,\beta]$ with the convex external field $U^{\me_{2}}(z)$, then $\sm \:\me_{1} \cap (-1, \beta)$ is an interval (see \cite[Theorem IV.1.10]{BookST}). By interchanging $\me_{1}$ and $\me_{2}$ and repeating the argument, we obtain that $\sm \: \me_{2} \cap (\beta, 1)$ is also an interval. It remains to show that $\alpha_{0} < \beta_{0}$. On the contrary, suppose that $\alpha_{0} = \beta = \beta_{0}$. Then, by (\ref{Frostman m1}) and (\ref{Frostman m2}), we have $U^{\me}(z) = F_{1} = F_{2}$ on $[-1,1]$, which implies $\me = \sigma$ and contradicts the assumption $q > m_{2}$.

		\textbf{ Step 2.} Finding the density function of the minimizing measure.
		
		 We are going to apply the standard technique of analyzing the Cauchy transform of the measure $\me$ (see, for example, \cite[Section VIII.5]{BookST}). Namely, consider the function 
		\beq \label{functH}
		H(z) = \int \frac{d \me (\zeta)}{z-\zeta}
		\eeq
		on the Riemann sphere $\mathbb{\overline{C}}$ cut along $\sm \: \me = [-1, \alpha_{0}] \cup [\beta_{0}, 1]$. The real part of $H$ vanishes on the cut since it is the derivative of the equilibrium potential of $\me$ and the latter is constant on $\sm \:\me$. 
		Since $H(\overline{z}) = \overline{H(z)}$ and $H(z)$ is continuous on $(-1, \alpha_{0}) \cup (\beta_{0},1)$, the function $H^{2}(z)$ is analytic on $(-1, \alpha_{0}) \cup (\beta_{0},1)$ and therefore analytic on $\mathbb{\overline{C}} \setminus \{-1, \alpha_{0}, \beta_{0}, 1\}$. Using Lemma \ref{AbsCont} one can easily verify that $H^{2}(z)$ can have at most simple poles at the points $\{-1, \alpha_{0}, \beta_{0}, 1\}$, and so $H^{2}(z)$ is a rational function on $\mathbb{\overline{C}}$. Moreover, because $H^{2}(z) \sim \frac{1}{z^{2}}$ when $z \rightarrow \infty$, we must have
		$$
		H^{2}(z) = \frac{(z-A)(z-B)}{(z +1)(z- \alpha_{0})(z-\beta_{0})(z-1)}, \;\;\; A, B \in \mathbb{R},
		$$  
		and thus
		$$
		H(z) = \bigg( \frac{(z-A)(z-B)}{(z +1)(z- \alpha_{0})(z-\beta_{0})(z-1)}  \bigg)^{1/2},
		$$
		where we take the branch of the square root that is positive for positive $z$.

		The function $H(z)$ on the upper part of the cut has the form
		$$
		H(z) = \frac{-i \sqrt{|z-A||z-B|}}{\sqrt{|(z +1)(z- \alpha_{0})(z-\beta_{0})(z-1)|}}, \;\;\; z \in \sm \:\me.  
		$$
	     For $z \in \overline{\mathbb{C}} \setminus \sm \:\me$ the Cauchy's formula gives 
		\beq \label{Hint}
		H(z) = \frac{1}{2 \pi i} \oint_{\sm \: \me } \frac{H(\zeta)}{\zeta - z} d \zeta= \frac{1}{\pi i} \int_{\sm \: \me } \frac{H(y) }{y - z} d y,
		\eeq
		where the second integral is taken on the upper part of the cut. Consequently, (\ref{functH}) together with (\ref{Hint}) imply
		\beq \label{density}
		d \me (y) = \frac{ \sqrt{|y-A||y-B|}}{\pi \sqrt{|(y +1)(y- \alpha_{0})(y-\beta_{0})(y-1)}|} dy.
		\eeq
		Our next goal is to show that in the expression (\ref{density}) of the density function $d \me(y)$ the constants $A, B$ must be equal to $\alpha_{0}$. Indeed, if we assume that $A \ne \alpha_{0} $ or $B \ne \alpha_{0}$ and then consider for $x \in (\alpha_{0}, \beta_{0})$
		$$
		\frac{d U^{\me}(x)}{dx} = -\frac{1}{\pi} \displaystyle\int\limits_{[-1, \alpha_{0}]\cup[\beta_{0},1]} \frac{1}{x-y} \frac{ \sqrt{|y-A||y-B|}}{\sqrt{|(y +1)(y- \alpha_{0})(y-\beta_{0})(y-1)}|} dy,
		$$
		it's easy to see that $\frac{d U^{\me}(x)}{dx} \rightarrow -\infty$ when $x  \rightarrow \alpha_{0}^{+}$. 
		If $\alpha_{0} < \beta$, then this contradicts (\ref{Frostman m1}), and we conclude that $A=B=\alpha_{0}$. Next, we show that the case $\alpha_{0} = \beta$ is impossible under the assumption $q > m_{2}$. To show this, assume the contrary, $\alpha_{0} = \beta$.
		Then it implies that $A, B$ in (\ref{density}) should be equal to $\beta_{0}$, since otherwise we would have $\frac{d U^{\me}(x)}{dx} \rightarrow \infty$ when $x  \rightarrow \beta_{0}^{-}$, that contradicts (\ref{Frostman m2}). Thus, the density function in this case is
		$$
		d \me (y) = \frac{ \sqrt{|y-\beta_{0}|}}{\pi \sqrt{(y +1)(y- \alpha_{0})(y-1)}} dy,
		$$
		and, since $\alpha_{0} \leq \beta_{0}$,
		$$
		\frac{1}{\pi} \int_{\beta_{0}}^{1} \frac{ \sqrt{|y-\beta_{0}|}}{\sqrt{(y +1)(y- \alpha_{0})(y-1)}} dy \leq
		\frac{1}{\pi} \int_{\beta_{0}}^{1} \frac{dy}{\sqrt{(y +1)(y-1)}} = m_{2},
		$$
		that contradicts the assumption $\me([\beta, 1]) = q > m_{2}$ and proves that we necessarily have $\alpha_{0}<\beta$ as well as $\beta_{0} = \beta$, that completes the proof.

	\end{proof}

	\begin{figure}[h!] 
		\center{\includegraphics[scale=0.6]{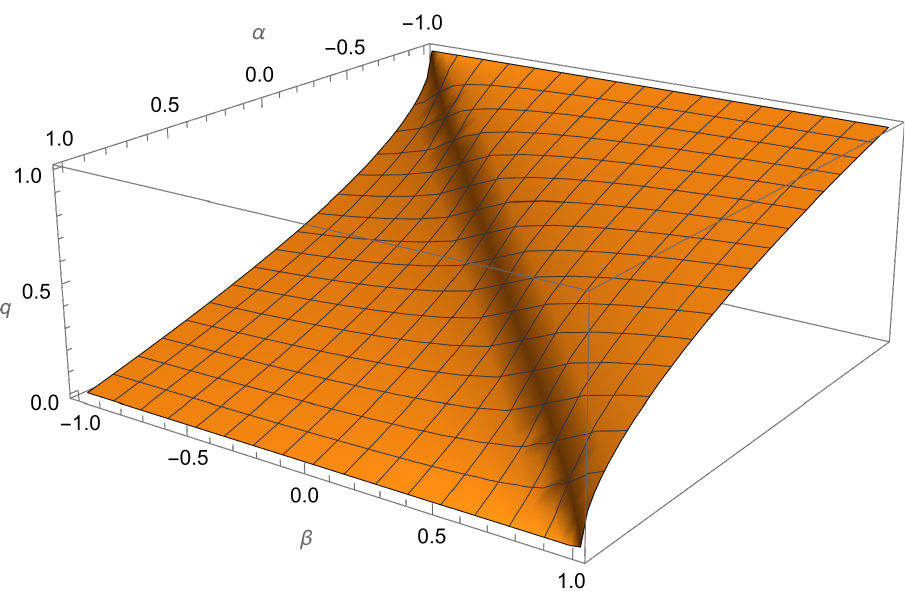}} 
		\caption{Graph showing the relationship between the parameters $\alpha$, $\beta$ and $q$. }
	\end{figure}

	\subsection{Proof of Theorem \ref{Sol I}} \label{Proof2}

	\begin{proof}
		For a measure $\mu \in \mathcal{P}([-1,1])$ with density $f(x)$, consider its logarithmic potential
		\beq \label{muPot}
		U^{\mu}(z) =  \int_{-1}^{1} \log \frac{1}{|z - x|} f(x) dx
		\eeq
		Consider the Joukowski map $z = \Psi(\zeta): = \frac{1}{2}(\zeta + \zeta^{-1})$ that maps the exterior of the unit circle, conformally to $\mathbb{C} \setminus [-1,1]$. We extend continuously the function $\Psi(\zeta)$ to the unit circle $\mathbb{S}^{1}$ and use the same notation for the extension. For $z = \Psi(\zeta) \in \mathbb{C} \setminus [-1,1]$ the potential (\ref{muPot}) can be written as
		\beq
		\begin{split} 
			U^{\mu}(z) &= \int_{-\pi}^{\pi} \log \frac{1}{|\Psi(\zeta) - \Psi(e^{i \psi})|} \widetilde{f}(\psi) d\psi \\
			&= \int_{-\pi}^{\pi} \log \bigg|\frac{\zeta - e^{i \psi}}{\Psi(\zeta) - \Psi(e^{i \psi})}\bigg| \widetilde{f}(\psi) d\psi + \int_{-\pi}^{\pi} \log \frac{1}{|\zeta - e^{i \psi}|} \widetilde{f}(\psi) d\psi,
		\end{split}
		\eeq
		where $x = e^{i \psi}$ and $\widetilde{f}$ is an even function on $\mathbb{S}^{1}$ with $\int_{\mathbb{S}^{1}} \widetilde{f} d\psi =1$.
		Let us denote $d\nu(e^{i \psi}):=\widetilde{f}(\psi)d\psi$. Then we have
		$$
		U^{\mu}(z) = h(\zeta) + U^{\nu}(\zeta), \;\; z = \Psi(\zeta) \in \mathbb{C} \setminus [-1,1].
		$$ 
		Since the function 
		$$g(\omega) = \log \bigg|\frac{\omega - \zeta}{\Psi(\omega) - \Psi(\zeta)}\bigg|$$
		 is harmonic in $\Delta:=\{z \in \mathbb{C}: |z| >1\}$ and continuous on the boundary, we have 
		\beq \label{Ucontin}
		\begin{split}
			\underset{\zeta \rightarrow e^{i \varphi}}{\lim} h(\zeta) &= \int_{-\pi}^{\pi} \log \bigg|\frac{e^{i \varphi} - e^{i \psi}}{\Psi(e^{i \varphi}) - \Psi(e^{i \psi})}\bigg| \widetilde{f}(\psi) d\psi \\&= \int_{-\pi}^{\pi} \log \frac{1}{|e^{-i \varphi} - e^{i \psi}|} \widetilde{f}(\psi) d\psi + \log 2.
		\end{split}
		\eeq
		
		If $U^{\mu}(z)$ is continuous in $\mathbb{C}$, then in view of (\ref{Ucontin}) we have
		
		\beq \label{limit}
		U^{\mu}(\Psi(e^{i \varphi}))  = \underset{z  \rightarrow \Psi(e^{i \varphi})}{\lim} U^{\mu}(z) = \underset{\zeta  \rightarrow e^{i \varphi}}{\lim} (h(\zeta) + U^{\nu}(\zeta)) = 2 U^{\nu}(e^{i \varphi}) + \log 2,
		\eeq     
		where the last equality follows from (\ref{Ucontin}) together with the fact that $U^{\nu}(\zeta) = U^{\nu}(\overline{\zeta})$, $\zeta \in \overline{\Delta}$.
		
		Now, if we take $\mu$ above to be the measure from (\ref{densitySol}) with $\beta = \cos(\frac{\theta}{2})$ and $q \geq \frac{1}{\pi} \int_{\beta}^{1} \frac{1}{\sqrt{1 - x^{2}}} dx = \frac{\theta}{2 \pi}$, then it's clear that the potential $U^{\me}(z)$ is continuous in $\mathbb{C}$, and
		$$
		U^{\me}(z) =  \frac{1}{2\pi} \int_{S} \log \frac{1}{|z - \Psi(e^{i \psi})|} \frac{\sqrt{|\cos(\psi) - \alpha|}}{\sqrt{|\cos(\psi) - \beta|}} d\psi,
		$$
		where the set $S := A_{\theta} \cup \{z \in \mathbb{S}^{1}: \arccos(\alpha) \leq \arg \; z \leq 2\pi -  \arccos(\alpha)\}$, and $\alpha$ is defined by (\ref{alpha2}). Consequently, (\ref{limit}) and \cite[Theorem VIII.2.1]{BookST} imply that the measure $\men$, where $d \men (e^{i \psi}) = \frac{1}{2 \pi} \frac{\sqrt{|\cos(\psi) - \alpha|}}{\sqrt{|\cos(\psi) - \beta|}} d \psi$, is the solution to the Problem II.
	\end{proof}

	\end{document}